\documentclass[12pt]{amsart}
\usepackage{fullpage,amssymb}
\newtheorem{theorem}{Theorem}[section]
\newtheorem{lemma}[theorem]{Lemma}
\newtheorem{corollary}[theorem]{Corollary}
\newtheorem{proposition}[theorem]{Proposition}
\theoremstyle{definition}
\newtheorem{definition}[theorem]{Definition}
\newtheorem{example}[theorem]{Example}
\newtheorem{remark}[theorem]{Remark}
\newcommand{\R}{{\mathbb R}}
\newcommand{\C}{{\mathbb C}}

\newcommand{\Z}{{\mathbb Z}}
\newcommand{\EE}{{\mathbb E}}
\newcommand{\PP}{{\mathbb P}}
\newcommand{\vol}{\operatorname{vol}}

\newcommand{\rank}{\operatorname{rank}}

\newcommand{\per}{\operatorname{per}}
\newcommand{\SSS}{{\mathbb S}}
\newcommand{\sgn}{\operatorname{sgn}}
\newcommand{\trace}{\operatorname{trace}}
\newcommand{\Hom}{\operatorname{Hom}}
\title{On the nuclear norm and the singular value decomposition of tensors}

\author{Harm Derksen}
\thanks{The author was partially supported by NSF grants DMS 0901298 and DMS 1302032.}

\begin{document}
\begin{abstract}
Finding the rank of a tensor is a problem that has many applications. Unfortunately it is often very difficult to determine the rank of a given tensor.
Inspired by the heuristics of convex relaxation, we consider the nuclear norm instead of the rank of a tensor. We  determine the nuclear
norm of various tensors of interest. Along the way, we also do a systematic study various measures of orthogonality in tensor product spaces and
we give a new generalization of the Singular Value Decomposition to higher order tensors.

\end{abstract}
\maketitle
\section{Introduction}\label{sec1}
\subsection{Tensor decompositions}
Suppose that $V=V^{(1)}\otimes \cdots\otimes V^{(d)}$ is the tensor product of finite dimensional Hilbert spaces.
For some applications, we would like to find a decomposition of  a given tensor $T$ as a sum of pure tensors:
\begin{equation}\label{eq:T}
T=\sum_{i=1}^r v_i, \mbox{ where }v_i=v_i^{(1)}\otimes v_i^{(2)}\otimes  \cdots \otimes v_i^{(d)}\mbox{ and } v_i^{(e)}\in V^{(e)}.
\end{equation}
 
 The smallest possible $r$ for which a decomposition (\ref{eq:T}) exists is called the {\em rank} of $T$ (see~\cite{Hitchcock}). For $d=2$, the rank of a tensor corresponds to the rank of a matrix.
So the tensor rank can be thought of as a generalization of the matrix rank to higher dimensional arrays.
For $d\geq 3$ it is difficult to determine
 the rank of a given tensor, or even to give good upper and lower bounds.  For example, a dimension counting argument shows that
 a dense open subset of $\C^n\otimes \C^n\otimes \C^n$ consists of tensors of rank $\geq n^3/(3n-2)=O(n^2)$. So far, there are no known explicit families of examples
 of tensors with a proven  lower bound of $\omega(n)$.
 The problem of finding the rank of a given tensor is known to be NP-hard (see~\cite{Hastad,Hastad2,HL}).
 The tensor rank plays an important role in Algebraic Complexity Theory. The complexity of matrix multiplication, for example, is closely related to the rank of a certain tensor (see Section~\ref{sec:matrix}) 
 
 In some applications, we just would like to find a low rank approximation: for a small fixed value of $r$, we want  to find pure tensors $v_1,\dots,v_r$
 such that the $\ell^2$-norm
$$
 \Big\|T-\sum_{i=1}^r v_i\Big\|
 $$
 is small. As pointed out in \cite{DSL}, there may not always be an optimal solution for which this norm is minimal.
  The problem of finding a low-rank approximation is known as the PARAFAC  (\cite{Harshman}) or CANDECOMP (\cite{CC})  model. There are many
 applications of this model, for example fluorescence spectroscopy,  statistics, psychometrics, geophysics and magnetic resonance imaging.

  \subsection{The nuclear and spectral norms}
  Convex relaxation is a powerful technique that is based on the following idea: Suppose that we are trying to find the sparsest solution $x=(x_1,\dots,x_n)\in \C^n$ to some problem. In other words, we are trying to find a solution $x$ such that 
  $$
  \|x\|_0:=|\{i\mid x_i\neq 0\}|
  $$
  is minimal.
  This is typically a very hard problem because the function $\|\cdot\|_0$ is not convex or continuous. But sometimes one can prove that minimizing the $\ell^1$-norm
  $$
  \|x\|_1=\sum_{i=1}^n |x_i|
  $$
   yields the sparsest solution. A sparse relaxation of the rank of a matrix $A$ is the nuclear norm $\|A\|_\star=\trace(\sqrt{AA^\star})$, which
   is also the sum of the singular values of $A$. In this context, this relaxation technique has been successfully applied to matrix completion problems in~\cite{RFP, CR, CT,KMO}.
 
 The nuclear norm can be  generalized to higher order tensors (see \cite[Definition~3.2]{LC}).
   The {\em nuclear norm} $\|T\|_\star$ of a tensor $T$ is the smallest possible value of $\sum_{i=1}^r \|v_i\|$ over all possible  decompositions (\ref{eq:T}). 
  The nuclear norm for tensors has been used for tensor completion problems in \cite{GRY}.
   
   The spectral norm  $[T]$ of $T$ is defined as the maximum value of $|\langle T,u\rangle|$ where $u$ ranges over all pure tensors of unit length.   For a matrix, the spectral norm is just the largest singular value.
   More generally, if ${\bf T}=(T_1,\dots,T_r)$ is an
   $r$-tuple of tensors, then we define $[{\bf T}]_{\alpha}$ as the maximum of
   $$
\Big( \sum_{i=1}^r| \langle T_i,u\rangle|^\alpha\Big)^{1/\alpha}
$$
over all pure tensors $u$ of unit length. The following theorem is useful for obtaining lower bounds for the spectral norm:
\begin{theorem}\label{theo:lowbound}
If $T$ is a tensor, ${\bf S}=(S_1,\dots,S_r)$ is an $r$-tuple of tensors  and $\alpha\geq 1$  then we have
$$
\Big( \sum_{i=1}^r |\langle T,S_i\rangle|^\alpha\Big)^{1/\alpha}\leq \|T\|_\star[{\bf S}]_\alpha
$$
\end{theorem}
The proof of Theorem~\ref{theo:lowbound} is in Section~\ref{sec5}.
If ${\bf S}=(S)$ just consists of a single tensor, then $[{\bf S}]_\alpha=[S]$ and we have:
   \begin{corollary}\label{cor:spectralbound}
   For tensors $S,T$ we have
   $$
   |\langle T,S\rangle|\leq \|T\|_\star [S].
   $$
   \end{corollary}
   Corollary~\ref{cor:spectralbound} can also easily be proven directly without using Theorem~\ref{theo:lowbound}. If we set $S=T$ then we obtain
   $$
   \|T\|^2\leq \|T\|_\star[T].
   $$

 \subsection{Singular Value Decomposition}
 The Singular Value Decomposition (SVD) can be generalized to higher-dimensional arrays. One such generalization was given in \cite{dLdMvdW}.
 Given a tensor $T$, one can choose an orthonormal bases $f_1^{(i)},\dots,f^{(i)}_{n_i}$ for $V_i$ for all $i$ and express $T$ in  these bases:
\begin{equation}\label{eq:TSVD}
T=\sum_{i_1=1}^{n_1}\sum_{i_2=1}^{n_2}\cdots\sum_{i_d=1}^{n_d}
\lambda_{i_1,i_2,\dots,i_d} f_{i_1}^{(1)}\otimes f_{i_2}^{(2)}\otimes \cdots\otimes f_{i_d}^{(d)}.
\end{equation}
Define
$$T^{(j)}_k=\sum_{i_j=k} \lambda_{i_1,i_2,\cdots,i_d}f_{i_1}^{(1)}\otimes f_{i_2}^{(2)}\otimes \cdots\otimes f_{i_d}^{(d)},
$$
where the sum runs over all $d$-tuples $(i_1,\dots,i_d)$ with $i_j=k$.
For a proper choice of the bases, the tensors $T_1^{(j)},\dots,T_{n_j}^{(j)}$ are orthogonal for all $j$
and
$$
\|T_1^{(j)}\|\geq \|T_2^{(j)}\|\geq \cdots\geq \|T_{n_j}^{(j)}\|.
$$
These numbers are called  the {\em singular values in mode $j$}. The decomposition~(\ref{eq:TSVD}) is called the {\em higher order single value decomposition} (HOSVD).

In this paper, we will give a different generalization of the SVD, which we call the {\em diagonal singular value decomposition} (DSVD). 
A given tensor may not have a diagonal singular value decomposition (see Section~\ref{sec7}),
but if it does, then the decomposition has many  nice properties.

\begin{definition}\label{def:tortho}
Suppose that $t\geq 1$ is a real number.
An $r$-tuple ${\bf S}=(S_1,\dots,S_r)$ of tensors of unit length  is called
 $t$-orthogonal if $[{\bf S}]_{2/t}=1$.
 \end{definition}
 If ${\bf v}=(v_1,\dots,v_r)$ is an $r$-tuple of pure  tensors of unit length,  then $t$-orthogonality implies orthogonality in the usual sense. 
Also, ${\bf v}$ is orthogonal if and only if it is $1$-orthogonal.
 \begin{definition}
If $\sigma_1\geq \sigma_2\geq \cdots\geq \sigma_r>0$ are real, and $(v_1,\dots,v_r)$ is a $2$-orthogonal $r$-tuple of pure tensors of unit length, then a decomposition
 
\begin{equation}\label{eq:TDSVD}
T=\sum_{i=1}^r\sigma_iv_i
\end{equation}
 is called a {\em diagonal singular value decomposition} (DSVD)
of $T$, and $\sigma_1,\dots,\sigma_r$ are called the {\em singular values} of $T$.
\end{definition}

For a tensor $T$ that has a diagonal singular value decomposition, we have the following results:
\begin{theorem}\label{theo:uniquesingval}
The singular values of $T$ are uniquely determined by $T$ (and do not depend on the choice of the  diagonal singular value decomposition).
\end{theorem}
\begin{theorem}\label{theo:properties}
Suppose that $T$ has a singular value decomposition with singular values $\sigma_1\geq \sigma_2\geq \cdots\geq \sigma_r>0$.  Then we have
$$
\|T\|_\star=\sum_{i=1}^r \sigma_i,\quad [T]=\sigma_1,\mbox{and }\|T\|=\sqrt{\textstyle\sum_{i=1}^r\sigma_i^2}.
$$
\end{theorem}
\begin{theorem}\label{theo:distinctsingvalues}
If the singular values of $T$ are distinct, then the diagonal singular value decomposition is unique.
\end{theorem}
\begin{theorem}\label{theo:torthoSVD}
If $T=\sum_{i=1}^r\sigma_i v_i$ is a diagonal singular value decomposition and $(v_1,\dots,v_r)$ is $t$-orthogonal for some $t>2$, then 
the diagonal singular value decomposition of $T$ is unique.
\end{theorem}
The proofs of Theorems~\ref{theo:uniquesingval}--\ref{theo:torthoSVD} are in Section~\ref{sec6}.
\subsection{Tensors and multi-linear maps}
To a tensor
$$
T=\sum_{i=1}^r v_i^{(1)}\otimes \cdots \otimes v_i^{(d)}\in V^{(1)}\otimes \cdots \otimes V^{(d)}
$$
we can associate a multilinear map 
$$\varphi_T:(V^{(1)})^\star \times (V^{(2)})^\star \times \cdots \times (V^{(d-1)})^\star \to V^{(d)}.
$$
defined by
$$
\varphi_T(f^{(1)},f^{(2)},\dots,f^{(d-1)})=\sum_{i=1}^r  \big(\textstyle\prod_{j=1}^{d-1}f^{(j)}(v_i^{(j)})\big)v^{(d)}_i.
$$
We will apply this correspondence to matrix multiplication.
\subsection{Matrix multiplication}\label{sec:matrix}
Let $\C^{p\times q}$ denote the set of $p\times q$ matrices. The Hermitian form is given by $\langle A,B\rangle=\trace(AB^\star)$. The matrix with a $1$ in position $(i,j)$ and zeroes everywhere else is denoted by $e_{i,j}$.
Then matrix multiplication
$$
\C^{p\times q}\times \C^{q\times r}\to \C^{p\times r}
$$
corresponds to the tensor
$$
\sum_{i=1}^p\sum_{j=1}^q\sum_{k=1}^r  e_{i,j}\otimes e_{j,k}\otimes e_{i,k}\in 
\C^{p\times q}\otimes \C^{q\times r}\otimes \C^{p\times r}.
$$
If we identify $\C^{p\times r}$ with $\C^{r\times p}$ then the tensor has the following, more symmetric, form:
\begin{equation}\label{eq:Mpqr}
M_{p,q,r}=\sum_{i=1}^p\sum_{j=1}^q\sum_{k=1}^r \in e_{i,j}\otimes e_{j,k}\otimes e_{k,i}\in 
\C^{p\times q}\otimes \C^{q\times r}\otimes \C^{r\times p}.
\end{equation}
From this formula it is clear that $\rank(M_{p,q,r})\leq pqr$. Strassen proved that $\rank(M_{2,2,2})\leq 7$ (see~\cite{Strassen}) by
giving a decomposition
\begin{multline}\label{eq:Strassen}
M_{2,2,2}=(e_{1,1}+e_{2,2})\otimes (e_{1,1}+e_{2,2})\otimes (e_{1,1}+e_{2,2})+(e_{2,1}-e_{2,2})\otimes e_{1,1}\otimes (e_{1,2}+e_{2,2})+\\(e_{1,2}+e_{2,2})\otimes  (e_{2,1}-e_{2,2})\otimes e_{1,1}+(e_{1,1}+e_{2,1})\otimes (e_{1,2}-e_{1,1})\otimes e_{2,2}+(e_{1,2}-e_{1,1})\otimes e_{2,2}\otimes (e_{1,1}+e_{2,1})+\\+e_{2,2}\otimes (e_{1,1}+e_{2,1})\otimes (e_{1,2}-e_{1,1})+
e_{1,1}\otimes (e_{1,2}+e_{2,2})\otimes (e_{2,1}-e_{2,2}).
\end{multline}

and used this to show
that two $n\times n$ matrices can be multiplied by using only $O(n^{\log_2(7)})$ arithmetic where $\log_2(7)\approx 2.81<3$. 
The usual way of multiplying two matrices takes $O(n^3)$ arithmetic operations. More generally, define
$$
\omega=\inf \Big\{\frac{\log(\rank(M_{p,q,r}))}{\log(pqr)}\Big| p,q,r\geq 2\Big\}.
$$
If $\varepsilon>0$, then  two $n\times n$ matrices can be multiplied using only $o(n^{\omega+\varepsilon})$ arithmetic operations (see \cite{BCLR} and \cite{BCA}).
Coppersmith and Winograd proved that $\omega<2.376$ in \cite{CW}. Only recently, this bound was improved by Stothers (\cite{Stothers}) to $\omega<2.3737$ and the current record is $\omega<2.3727$ by Williams (\cite{Williams}).

For most values of $p,q,r$ the rank of $M_{p,q,r}$ is unknown. It is easy to see that  $\rank(M_{n,n,n})\geq n^2$. Bl\"aser gave a better, nontrivial lower bound in \cite{Blaser}.
A sharper lower bound was given by Landsberg in \cite{Landsberg2}, and using the same techniques, 
Massarenti and Raviolo (see~\cite{MR}) improved this lower bound to
$$
\rank(M_{n,n,n})\geq 3n^2-2\sqrt{2}n^{3/2}-3n.
$$
\begin{theorem}\label{theo:Mpqr}
The decomposition (\ref{eq:Mpqr}) is a diagonal singular value decomposition. In particular, the singular values of $M_{p,q,r}$ are
$$
\underbrace{1,1,\dots,1}_{pqr}.
$$
\end{theorem}
The proof of the theorem is in Section~\ref{sec4}. The following corollary follows from  Theorem~\ref{theo:Mpqr} and Theorem~\ref{theo:properties}.
\begin{corollary}
We have $\|M_{p,q,r}\|_\star=pqr$ and $[M_{p,q,r}]=1$.
\end{corollary}

Note that the sum of the lengths of the  pure tensors in the decomposition (\ref{eq:Strassen}) is $2\sqrt{2}+12>8$. This shows that minimization of the rank, and minimization of the nuclear norm do not always coincide.
\subsection{The discrete Fourier transform and group algebras}
Suppose that $G$ is a finite group. The group algebra $\C G$ is the vector space with a orthonormal basis $g,g\in G$.
Multiplication in the group $G$ gives $\C G$ the structure of an associative algebra. 
The multiplication
$$
\C G\times \C G\to \C G
$$
corresponds to the tensor
$$
\sum_{g\in G}\sum_{h\in G} g\otimes h\otimes gh\in \C G\otimes \C G\otimes \C G.
$$
By permuting the basis vectors in the last factor $\C G$, the tensor can be written in the following symmetric form:
$$
T_G:=\sum_{\scriptstyle g,h,k\in G\atop  \scriptstyle ghk=1} g\otimes h\otimes k.
$$
\begin{theorem}\label{theo:groupalgebra}
If $G$ is a group of order $n$, $d_1,\dots,d_s$ are the dimensions of the irreducible representations of $G$, then the tensor $T_G$ has a diagonal singular value decomposition and its singular values are
$$
 \underbrace{\textstyle\sqrt{\frac{n}{d_1}},\dots,\sqrt{\frac{n}{d_1}} }_{d_1^3},\underbrace{\textstyle \sqrt{\frac{n}{d_2}},\dots,\sqrt{\frac{n}{d_2}}}_{d_2^3},\dots,
  \underbrace{\textstyle\sqrt{\frac{n}{d_s}},\dots,\sqrt{\frac{n}{d_{s}}}  }_{d_s^3}.
$$
\end{theorem}
The proof of Theorem~\ref{theo:groupalgebra} can be found in Section~\ref{sec4}. From Theorem~\ref{theo:groupalgebra} and Theorem~\ref{theo:properties}
we get the following result.
\begin{corollary}
We have  $\|T_G\|_\star=\sqrt{n}\sum_{i=1}^s d_i^{5/2}$ and $[T_G]=\sqrt{n}$.
\end{corollary}

Let $C_n$ be the (multiplicative) cyclic group of order $n$  generated by $x$. Then $\C C_n$ is the commutative ring $\C[x]/(x^n-1)$ and multiplication
in $\C G$ corresponds to the multiplication of polynomials in one variable (modulo $x^n-1$).
 We have
\begin{equation}\label{eq:TCn}
T_{C_n}=\sum_{i+j+k=0} x^i\otimes x^j\otimes x^k
\end{equation}
where the sum is over all $i,j,k\in \Z/n\Z$ with $i+j+k=0$. From (\ref{eq:TCn}) follows that $\rank(T_{C_n})\leq n^2$ and $\|T_{C_n}\|_\star \leq n^2$.
Let $\zeta=e^{2\pi i/n}$ be a primitive $n$-th root of unity.
The Discrete Fourier Transform
is based on the following decomposition of $T_{C_n}$:
\begin{equation}\label{eq:DFT}
T_{C_n}=\sum_{t=0}^{n-1} \sqrt{n}\textstyle (\frac{1}{\sqrt{n}}\sum_{i=0}^{n-1}\zeta^{t i}x^i)\otimes (\frac{1}{\sqrt{n}}\sum_{j=0}^{n-1}\zeta^{t j}x^j)\otimes (\frac{1}{\sqrt{n}}\sum_{k=0}^{n-1}\zeta^{t k}x^k).
\end{equation}
The following Theorem follows from Theorems~\ref{theo:groupalgebra} and~\ref{theo:torthoSVD}.
\begin{theorem}
The decomposition (\ref{eq:DFT}) is the unique diagonal singular value decomposition. In particular, the singular values of $T_{C_n}$ are
$$
\underbrace{\sqrt{n},\sqrt{n},\dots,\sqrt{n}}_n
$$
and $\|T_{C_n}\|_\star=n\sqrt{n}$.
\end{theorem}
\subsection{The determinant and the permanent}
The determinant and permanent are multilinear functions 
$$\underbrace{\C^n\times \cdots \times \C^n}_n\to \C.
$$
Let $\SSS_n$ be the symmetric group on $n$ letters. The sign of a permutation $\sigma\in \SSS_n$ is denoted by $\sgn(\sigma)\in \{1,-1\}$.
The determinant and permanent correspond to the tensors
$$
{\textstyle \det_n}:=\sum_{\sigma\in \SSS_n} \sgn(\sigma) e_{\sigma(1)}\otimes e_{\sigma(2)}\otimes \cdots\otimes e_{\sigma(n)}
$$
and 
$$
{\textstyle \per_n}:=\sum_{\sigma\in \SSS_n}e_{\sigma(1)}\otimes e_{\sigma(2)}\otimes \cdots\otimes e_{\sigma(n)}
$$
respectively in 
$$
\underbrace{\C^{n}\otimes \cdots \otimes \C^n}_n\otimes \C\cong \underbrace{ \C^{n}\otimes \cdots \otimes \C^n}_n.
$$
From these formulas it is clear that $\rank(\det_n)\leq n!$ and $\rank(\per_n)\leq n!$.
The upper bound for the rank of the determinant is not sharp for $n\geq 3$ (see Section~\ref{sec8}).
The bound for the permanent is far from optimal. Another formula for
the permanent was given by Glynn~\cite{Glynn}:
\begin{equation}
\label{eq:Glynn}
\per_n=\frac{1}{2^{n-1}} \sum_{\delta}\Big(\prod_{k=1}^n \delta_k\Big)\textstyle (\sum_{j=1}^n \delta_je_j)\otimes (\sum_{j=1}^n \delta_je_j)\otimes \cdots \otimes (\sum_{j=1}^n \delta_je_j)
\end{equation}
where $\delta$ runs over all $2^{n-1}$ vectors $\delta=(\delta_1,\dots,\delta_n)\in \{1,-1\}^n$ with $\delta_1=1$. From this formula follows that $\rank(\per_n)\leq 2^{n-1}$
and $\|\per_n\|_\star\leq n^{n/2}$. Some easy lower bounds for the rank of the the permanent 
and determinant are given in Section~\ref{sec8}.
\begin{theorem}\label{theo:perm}
We have $\|\per_n\|_\star=n^{n/2}$.
\end{theorem}
The formula (\ref{eq:Glynn}) minimizes the sum of the lengths of the pure tensors, but these pure tensors are not $2$-orthogonal (or even orthogonal) for $n\geq 3$. 
In fact, for $d\geq 3$ the tensor $\per_n$ does not have a diagonal singular value decomposition (see Section~\ref{sec7}).
\begin{theorem}\label{theo:det}
We have $\|\det_n\|_\star=n!$.
\end{theorem}

The proofs of Theorems~\ref{theo:det} and  \ref{theo:perm} are in Section~\ref{sec5}.

\section{Orthogonality of vectors}\label{sec2}
In this section we will study various measures of orthogonality of vectors in a Hilbert space $V$.
It is  convenient to deal with unit vectors. Suppose that ${\bf v}=(v_1,\dots,v_r)$ is an $r$-tuple of unit vectors.
\begin{definition}
The {\em coherence}  is defined by
$$
\mu({\bf v})=\max_{i\neq j}|\langle v_i,v_j\rangle|.
$$
\end{definition}

More generally, we will define
\begin{definition}
$$
\mu_{\alpha}({\bf v})=\max_i\left\{\Big(\sum_{j\neq i} |\langle v_i,v_j\rangle|^{\alpha}\Big)^{1/\alpha}\right\}.
$$
\end{definition}
Notice that $\lim_{\alpha\to\infty}\mu_\alpha({\bf v})=\mu({\bf v})$. So we may think of $\mu({\bf v})$ as $\mu_\infty({\bf v})$.

Suppose that ${\bf v}=(v_1,\dots,v_r)$ and ${\bf w}=(w_1,\dots,w_r)$ are $r$-tuples of unit vectors. We define the {\em horizontal tensor product} of ${\bf v}$ and ${\bf w}$ by
$$
{\bf v}\otimes {\bf w}:=(v_1\otimes w_1,\dots,v_r\otimes w_r)\in (V\otimes W)^r
$$
If ${\bf v}=(v_1,\dots,v_r)\in V^r$ and ${\bf w}=(w_1,\dots,w_s)\in W^s$ then we define
$$
{\bf v}\boxtimes {\bf w}=(v_1\otimes w_1,v_1\otimes w_2,\dots,v_1\otimes w_s,v_2\otimes w_1,\dots,v_r\otimes w_s)\in (V\otimes W)^{rs}.
$$
It is easy to see that
$$
\mu(v\boxtimes w)=\max\{\mu(v),\mu(w)\}.
$$

\begin{lemma}
If ${\bf v}=(v_1,\dots,v_r)\in V^r$ and ${\bf w}=(w_1,\dots,w_s)\in W^s$ are  tuples of unit vectors, then we have
$$
\mu_{\alpha}({\bf v}\boxtimes {\bf w})^\alpha+1=(\mu_{\alpha}({\bf v})^\alpha+1)(\mu_\alpha({\bf w})^\alpha+1)
$$
for $\alpha>0$ and
$$
\mu({\bf v}\otimes {\bf w})=\max\{\mu({\bf v}),\mu({\bf w})\}.
$$
\end{lemma}
\begin{proof}
We
have
\begin{multline*}
\mu({\bf v}\boxtimes {\bf w})^\alpha_\alpha+1 =\max_{i,j}\left\{\sum_{k=1}^r\sum_{l=1}^s |\langle v_i\otimes w_j,v_k\otimes w_l\rangle|^{\alpha}\right\}=
\max_{i,j}\left\{\sum_{k=1}^r\sum_{l=1}^s |\langle v_i,v_k\rangle\langle w_j, w_l\rangle|^{\alpha}\right\}=\\
\max_{i,j}\left\{\Big(\sum_{k=1}^r |\langle v_i,v_k\rangle|^{\alpha}\Big)\Big(\sum_{l=1}^s |\langle w_j, w_l\rangle|^{\alpha}\Big)\right\}=\\=
\max_i\left\{\sum_{k=1}^r |\langle v_i,v_k\rangle|^{\alpha}\right\}\max_j\left\{\sum_{l=1}^s |\langle w_j, w_l\rangle|^{\alpha}\right\}=(\mu({\bf v})^\alpha_\alpha+1)(\mu({\bf w})^\alpha_\alpha+1).
\end{multline*}
\end{proof}
We will need a slightly more general version of the  H\"older inequality.
\begin{lemma}[H\"older inequality]
If $a_1,\dots,a_r,b_1,\dots,b_r$ are nonnegative real numbers, and $\alpha,\beta,\gamma$ are positive real numbers with $1/\alpha+1/\beta=1/\gamma$, then we have
$$
\Big(\sum_{i=1}^r (a_ib_i)^\gamma\Big)^{1/\gamma}\leq \Big(\sum_{i=1}^r a_i^\alpha\Big)^{1/\alpha}\Big(\sum_{i=1}^r b_i^\beta\Big)^{1/\beta}.
$$
\end{lemma}
\begin{proof}
The usual H\"older inequality states that, if $1/p+1/q=1$, then we have
$$
\sum_{i=1}^r a_ib_i\leq \Big(\sum_{i=1}^r a_i^p\Big)^{1/p}\Big(\sum_{i=1}^r b_i^q\Big)^{1/q}
$$
with equality if and only if  the vectors $(a_1^p,\dots,a_r^p)$ and $(b_1^q,\dots,b_r^q)$ are dependent.
Now take $p=\alpha/\gamma$ and $q=\beta/\gamma$, and replace $a_i$ and $b_i$ by $a_i^\gamma$ and $b_i^\gamma$ respectively:
$$
\sum_{i=1}^r  (a_ib_i)^\gamma\leq \Big(\sum_{i=1}^r a_i^\beta\Big)^{\gamma/\alpha}\Big(\sum_{i=1}^r b_i^\beta\Big)^{\gamma/\beta}.
$$
Taking the $\gamma$-th root gives the desired inequality.
\end{proof}
For horizontal tensor products we have a H\"older inequality:
\begin{lemma} \label{lem:Holderbraces}
If $1/\alpha+1/\beta=1/\gamma$, then we have
$$
\mu_\gamma({\bf v}\otimes {\bf w})\leq \mu_\alpha({\bf v})\mu_\beta({\bf w})
$$
\end{lemma}
\begin{proof}
By the H\"older inequality we have
\begin{multline*}
\Big(\sum_{j\neq i}|\langle v_i\otimes w_i,v_j\otimes w_j\rangle|^\gamma\Big)^{1/\gamma}\leq
\Big(\sum_{j\neq i}(|\langle v_i,v_j\rangle|\cdot |\langle w_i,w_j\rangle|)^\gamma\Big)^{1/\gamma}\leq\\ \leq
\Big(\sum_{j\neq i}|\langle v_i,v_j\rangle|^\alpha\Big)^{1/\alpha}\Big(\sum_{j\neq i}|\langle w_i,w_j\rangle|^\beta\Big)^{1/\beta}
\end{multline*}
for all $i$. Taking the maximum over all $i$ on both sides gives the desired inequality.
\end{proof}
If we take $\beta\to\infty$ we get the inequalities
$$
\mu_{\alpha}({\bf v}\otimes {\bf w})\leq \mu_{\alpha}({\bf v})\mu({\bf w})
$$
and 
$$
\mu({\bf v}\otimes {\bf w})\leq \mu({\bf v})\mu({\bf w}).
$$

\begin{lemma}
If $\gamma>\alpha>0$ then and ${\bf v}=(v_1,\dots,v_r)$ is an $r$-tuple of unit vectors, then we have
$$
\mu_\gamma({\bf v})^{\gamma/\alpha}\leq\mu_\alpha({\bf v})\leq \mu_\gamma({\bf v})(r-1)^{1/\alpha-1/\gamma}.
$$
\end{lemma}
\begin{proof}
The inequality on the left is easy. For the inequality on the right, let $w=(1,1,\dots,1)$
and identify $v\otimes w\in (V\otimes \C)^r\cong V^r$ with $v$. Then apply Lemma~\ref{lem:Holderbraces}.
\end{proof}
If we take the limit $\gamma\to\infty$ we get
$$
0 \leq\mu_\alpha({\bf v})\leq \mu({\bf v})(r-1)^{1/\alpha}.
$$

For an $r$-tuple ${\bf v}=(v_1,\dots,v_r)$ of vectors we define
$$
{\bf v}^{\otimes d}=\underbrace{{\bf v}\otimes {\bf v}\otimes \cdots\otimes {\bf v}}_d.
$$
\begin{lemma}
For an $r$-tuple of unit vectors ${\bf v}$ we have $\mu_\alpha({\bf v}^{\otimes d})=\mu_{d\alpha}({\bf v})^d$.
\end{lemma}
\begin{proof}
We have
$$
\Big(\sum_{j\neq i}|\langle v_i^{\otimes d},v_j^{\otimes d}\rangle|^{\alpha}\Big)^{1/\alpha}=
\Big(\sum_{j\neq i}|\langle v_i,v_j\rangle|^{d\alpha}\Big)^{1/\alpha}
$$
Taking the maximum over all $i$ on both sides gives the desired result.
\end{proof}
\section{Orthogonality of tensors}\label{sec3}
In this section, we study another measure for the orthogonality of pure tensors, which takes into account the tensor product structure of the vector space. It is important, when discussing pure tensors, to be clear which tensor product structure we are talking about. To be unambiguous, we make the following definition.
An {\em$d$-th order tensor space} is a pair ${\bf V}=(V,(V^{(1)},\dots,V^{(d)}))$
where $V^{(1)},\dots,V^{(d)}$ are finite dimensional Hilbert spaces and 
 $$V=V^{(1)}\otimes V^{(2)}\otimes \cdots\otimes V^{(d)}.$$ 
 A {\em pure tensor} (with respect to this tensor space) is an element in $V$ of the form 
 $$v^{(1)}\otimes v^{(2)}\otimes \cdots\otimes v^{(d)}$$
 with $v^{(i)}\in V^{(i)}$ for $i=1,2,\dots,d$.
Suppose that ${\bf V}=(V,(V^{(1)},\dots,V^{(d)}))$ and ${\bf W}=(W,(W^{(1)}\dots,W^{(e)}))$ are tensor product spaces.
Then their {\em horizontal tensor product} is the tensor space
$$
{\bf V}\otimes {\bf W}:=(V\otimes W,(V^{(1)},\dots,V^{(d)},W^{(1)},\dots,W^{(e)})).
$$
If ${\bf S}=(S_1,\dots,S_r)\in {\bf V}^r$ and ${\bf T}=(T_1,\dots,T_r)\in {\bf W}^r$
then we define
$$
{\bf S}\otimes {\bf T}:=(S_1\otimes T_1,\dots,S_r\otimes T_r)\in ({\bf V}\otimes {\bf W})^r.
$$
For the horizontal tensor product we have a H\"older inequality.
\begin{lemma}\label{lem:Holder}
If $1/\alpha+1/\beta=1/\gamma$, then we have
$$
[{\bf S}\otimes {\bf T}]_\gamma\leq [{\bf T}]_\alpha[{\bf S}]_\beta.
$$
\end{lemma}
\begin{proof}
Using the H\"older inequality, we get
\begin{multline*}
\Big(\sum_{i=1}^r |\langle S_i\otimes T_i,x\otimes y\rangle |^\gamma\Big)^{1/\gamma}=
\Big(\sum_{i=1}^r |\langle S_i, x\rangle|^\gamma |\langle T_i, y\rangle|^\gamma\Big)^{1/\gamma}
\leq\\ \leq\Big(\sum_{i=1}^n|\langle S_i, x\rangle|^\alpha\Big)^{1/\alpha}\Big(\sum_{i=1}^n |\langle T_i, y|\rangle|^\beta\Big)^{1/\beta}.
\end{multline*}
Taking the supremum over all unit pure tensors $x$ and $y$ yields
the desired inequality.
\end{proof}
\begin{corollary}\label{cor:mtensor}
We have
$$
[{\bf S}^{\otimes d}]_{\alpha}=[{\bf S}]_{d\alpha}^d.
$$
\end{corollary}
\begin{proof}
For some unit pure tensor $u$ we have 
$$
[{\bf S}]_{d\alpha}^d=\Big(\sum_{i=1}^r |\langle S_i,u\rangle|^{d\alpha}\Big)^{1/\alpha}=
\Big(\sum_{i=1}^r |\langle S_i^{\otimes d},u^{\otimes d}\rangle|^{\alpha}\Big)^{\alpha}\leq [{\bf S}^{\otimes d}]_\alpha.
$$
The inequality in the other direction follows from Lemma~\ref{lem:Holder}.
\end{proof}
If ${\bf V}=(V,(V^{(1)},\dots,V^{(d)})$ and ${\bf W}=(W,(W^{(1)},\dots,W^{(d)}))$ are tensor product spaces, then 
their {\em vertical tensor product} is 
$$
{\bf V}\boxtimes {\bf W}:=(V\otimes W,(V^{(1)}\otimes W^{(1)},\dots,V^{(d)}\otimes W^{(d)})).
$$
If $S\in {\bf V}$ and $T\in {\bf W}$ we define $S\boxtimes T$ as $S\otimes T$,
viewed inside the tensor product space  ${\bf V}\boxtimes {\bf W}$.
If ${\bf S}\in {\bf V}^r$ and ${\bf T}\in {\bf W}^s$, then we define
$$
{\bf S}\boxtimes {\bf T}=(S_i\boxtimes T_j\mid 1\leq i\leq r,1\leq j\leq s)\in ({\bf V}\boxtimes {\bf W})^{rs}
$$

The measure  $[-]_\alpha$ also behaves multiplicatively with respect to the vertical tensor product.
\begin{proposition}\label{prop:verticaltensor}
Suppose that ${\bf V}=(V,(V^{(1)},\dots,V^{(d)}))$ and ${\bf W}=(W,(W^{(1)},\dots,W^{(d)}))$ are tensor product spaces, and $S\in {\bf V}^r$ and $T\in {\bf W}^s$.
Then we have
$$
[S\boxtimes T]_\alpha=[S]_\alpha[T]_\alpha.
$$
\end{proposition}
\begin{proof}
First, we will assume that $\alpha\leq 1$.
For a complex vector $b=(b_1,\dots,b_l)$ we have
\begin{equation}\label{eq:alphaineq}
\Big|\sum_{k=1}^l b_k\Big|^\alpha\leq \Big(\sum_{k=1} ^l|b_k|\Big)^\alpha=
\|b\|_1^\alpha\leq \|b\|_\alpha^\alpha=
 \sum_{k=1}^l |b_k|^\alpha.
\end{equation}
Suppose that $u$ is a pure tensor in ${\bf V}\boxtimes {\bf W}$. We can write 
$$
u=u^{(1)}\otimes u^{(2)}\otimes \cdots\otimes u^{(d)}
$$
with $u^{(e)}\in Z^{(e)}=V^{(e)}\otimes W^{(e)}$ and $\|u^{(e)}\|=1$ for all $e$.
Using the singular value decomposition, we can write
$$
u^{(e)}=\sum_{k}\lambda^{(e)}_k (x^{(e)}_k\boxtimes y^{(e)}_k)\in V^{(e)}\boxtimes W^{(e)}
$$
where $x_1^{(e)},x_2^{(e)},\dots$ and $y_1^{(e)},y_2^{(e)},\dots$ are (finite) sequences of orthonormal vectors for all $e$, 
 and $\lambda_k^{(e)}>0$ for all $k,e$. Since $u^{(e)}$ is a unit vector, we have $\sum_{k}(\lambda_k^{(e)})^2=1$.
We define $x^{(e)}=\sum_{k} \lambda_k^{(e)}$, $y^{(e)}=\sum_k\lambda_k^{(e)}$ for all $e$, and 
$$
x=x^{(1)}\otimes \cdots \otimes x^{(d)},\quad y=y^{(1)}\otimes \cdots \otimes y^{(d)}.
$$
Note that $x^{(e)}$ and $y^{(e)}$ are unit vectors, and $x$ and $y$ are pure tensors of unit length.
For a $d$-tuple $\underline{k}=(k_1,\dots,k_d)$ we define
$$
\lambda_{\underline{k}}=\prod_{e=1}^d \lambda_{k_e}^{(e)},\quad x_{\underline{k}}=x_{k_1}^{(1)}\otimes \cdots\otimes x_{k_d}^{(d)},\quad
y_{\underline{k}}=y_{k_1}^{(1)}\otimes\cdots \otimes y_{k_d}^{(d)}.
$$
So we can write
$$
u=\sum_{\underline{k}}\lambda_{\underline{k}} (x_{\underline{k}}\boxtimes  y_{\underline{k}}).
$$
This is a singular value decomposition of $u$, if $u$ is viewed as a tensor in 
$$(V\otimes W, (V^{(1)}\otimes \cdots\otimes V^{(d)},W^{(1)}\otimes \cdots\otimes W^{(d)})).
$$

Using the inequality (\ref{eq:alphaineq}) and the Cauchy-Schwarz inequality, we get
\begin{multline*}
\sum_{i,j} |\langle S_i\boxtimes T_j,u\rangle|^\alpha=\sum_{i,j}\Big|\sum_{\underline{k}}\lambda_{\underline{k}}\langle S_i\boxtimes T_j,x_{\underline{k}}\boxtimes y_{\underline{k}}\rangle\Big|^\alpha\leq
\sum_{i,j}\sum_{\underline{k}} \Big|\lambda_{\underline{k}}\langle S_i\boxtimes T_j,x_{\underline{k}}\boxtimes y_{\underline{k}}\rangle\Big|^\alpha=\\=
\sum_{\underline{k}}\lambda_{\underline{k}}^\alpha\sum_{i,j} |\langle S_i,x_{\underline{k}}\rangle|^\alpha |\langle T_j y_{\underline{k}}\rangle|^\alpha\leq
\sum_{\underline{k}}\Big(\lambda_{\underline{k}}^{\alpha/2}\sum_i|\langle S_i,x_{\underline{k}}\rangle|^\alpha\Big)\Big(\lambda_{\underline{k}}^{\alpha/2}\sum_j|\langle T_j,y_{\underline{k}}\rangle|^\alpha\Big)\leq\\
\leq \sqrt{\textstyle \sum_{\underline{k}}\lambda_{\underline{k}}^{\alpha}\Big(\sum_i|\langle S_i,x_{\underline{k}}\rangle|^{\alpha}\Big)^2} \sqrt{\textstyle \sum_{\underline{k}}
\lambda_{\underline{k}}^{\alpha}\Big(\sum_j|\langle T_j,y_{\underline{k}}\rangle|^{\alpha}\Big)^2}\leq\\
\leq \sqrt{\textstyle \sum_{\underline{k}}\lambda_{\underline{k}}^{\alpha}[{\bf S}]_\alpha^\alpha\sum_i|\langle S_i,x_{\underline{k}}\rangle|^{\alpha}} \sqrt{\textstyle \sum_{\underline{k}}
\lambda_{\underline{k}}^{\alpha}[{\bf T}]_{\alpha}^\alpha\sum_j|\langle T_j,y_{\underline{k}}\rangle|^{\alpha}}\leq\\
\leq [{\bf S}]_\alpha^{\alpha/2}[{\bf T}]_\alpha^{\alpha/2}\sqrt{\textstyle\sum_i |\langle S_i,\sum_{\underline{k}}\lambda_{\underline{k}} x_{\underline{k}}\rangle|^\alpha}
\sqrt{\textstyle\sum_j |\langle T_j,\sum_{\underline{k}}\lambda_{\underline{k}} y_{\underline{k}}\rangle|^\alpha}\leq\\
\leq [{\bf S}]_\alpha^{\alpha/2}[{\bf T}]_\alpha^{\alpha/2}\sqrt{\textstyle\sum_i |\langle S_i,x\rangle|^\alpha}
\sqrt{\textstyle\sum_j |\langle T_j,y\rangle|^\alpha}\leq
  [{\bf S}]_\alpha^{\alpha/2}[{\bf T}]_\alpha^{\alpha/2} [{\bf S}]_\alpha^{\alpha/2}[{\bf T}]_\alpha^{\alpha/2}= [{\bf S}]_\alpha^{\alpha}[{\bf T}]_\alpha^{\alpha}
\end{multline*}
Since the pure tensor $u$ was arbitrary, we have
$[{\bf S}\boxtimes {\bf T}]_\alpha^\alpha\leq[{\bf S}]_\alpha^\alpha [{\bf T}]_\alpha^\alpha$ and
$[{\bf S}\boxtimes {\bf T}]_\alpha\leq [{\bf S}]_\alpha[{\bf T}]_\alpha$.

If $\alpha\geq 1$, choose $m$ such that $\alpha/m\leq 1$.
By Corollary~\ref{cor:mtensor} we have
$$
[({\bf S}\boxtimes {\bf T})]_{\alpha}^m=[({\bf S}\boxtimes {\bf T})^{\otimes m}]_{\alpha/m}=[{\bf S}^{\otimes m}\boxtimes {\bf T}^{\otimes m}]_{\alpha/m}\leq
 [{\bf S}^{\otimes m}]_{\alpha/m}[{\bf T}^{\otimes m}]_{\alpha/m}=[{\bf S}]_{\alpha}^m[{\bf T}]_\alpha^m, 
$$
so we get $[{\bf S}\otimes {\bf T}]_\alpha\leq [{\bf S}]_\alpha [{\bf T}]_\alpha$.

Suppose that $\alpha>0$. There exists unit pure tensors $a\in {\bf V}$ and $b\in {\bf W}$ such that
$$
\sum_{i}|\langle S_i,a\rangle|^\alpha=[{\bf S}]_\alpha^\alpha\mbox{ and }\sum_j \langle T_j,b\rangle|^\alpha=[{\bf T}]^\alpha_\alpha.
$$
We have
$$
\sum_{i,j}|\langle S_i\boxtimes T_j,a\boxtimes b\rangle|^\alpha=\sum_{i,j} |\langle S_i,a\rangle|^\alpha |\langle T_j,b\rangle|^\alpha=
\sum_{i}|\langle S_i,a\rangle|^\alpha\sum_{j}|\langle T_j,b\rangle|^\alpha=[{\bf S}]_\alpha^\alpha [{\bf T}]_\alpha^\alpha.
$$
So it follows that $[{\bf S}\boxtimes {\bf T}]_\alpha^\alpha\geq [{\bf S}]_\alpha^\alpha[{\bf T}]_\alpha^\alpha$.
We conclude that $[{\bf S}\boxtimes {\bf T}]_\alpha=[{\bf S}]_\alpha[{\bf T}]_\alpha$.

\end{proof}

The norm $[-]_\alpha$ is hard to compute in practice because we have to solve a optimization problem.
But $\mu_\alpha(-)$ is easier to compute.
Fortunately,  $[-]_\alpha$ can be estimated in terms of $\mu$:

\begin{lemma}\label{lem:comparemeasures}
For an $r$-tuple ${\bf v}=(v_1,\dots,v_r)$ of pure tensors of unit length and $\alpha>0$ we have
$$
\mu_{\alpha}({\bf v})^\alpha+1\leq [{\bf v}]_\alpha^\alpha.
$$
\end{lemma}
\begin{proof}
Suppose that $\alpha>0$. For every $i$ we have
$$
1+\sum_{j\neq i}|\langle v_i,v_j\rangle|^\alpha= \sum_{j=1}^r|\langle v_i,v_j\rangle|^\alpha\leq [{\bf v}]_{\alpha}^\alpha.
$$
Taking the maximum over all $i$ gives the desired inequality.
\end{proof}

\begin{proposition}\label{prop:comparemeasures}
 For an $r$-tuple  ${\bf v}=(v_1,\dots,v_r)$ of pure tensors of unit length and $\alpha\geq 1$ we have
$$
 [{\bf v}]_{2\alpha}^{2\alpha}\leq \mu_{\alpha}({\bf v})^{\alpha}+1.
$$
\end{proposition}
\begin{proof}

Suppose that $\alpha\geq 2$.
Choose a unit pure tensor $w$ such that 
$$
[{\bf v}]_{2\alpha}=\Big(\sum_{i=1}^r |\langle v_i,w\rangle|^{2\alpha}\Big)^{1/(2\alpha)}.
$$
Let $D$ be the $r\times r$ diagonal matrix with $D_{i,i}=|\langle v_i,w\rangle|^{2\alpha-2}$
and define
$$Z=\begin{pmatrix} v_1 & \cdots & v_r\end{pmatrix},
$$
 where $v_1,\dots,v_r$ are viewed as
column vectors with respect to some orthonormal basis.
Consider the Hermitian matrix
$$
A=\sum_{i=1}^r |\langle v_i,w\rangle|^{2\alpha-2}v_iv_i^\star=ZDZ^\star
$$
Then we have $w^\star A w=[{\bf v}]_{2\alpha}^{2\alpha}$. If $\lambda\in \R$ is the largest eigenvalue of $A$, then we get
 $[{\bf v}]_{2\alpha}^{2\alpha}\leq \lambda$.
 The largest eigenvalue of the matrix $B=Z^\star ZD$ is $\lambda$ as well.
 Let
 $$
 x=\begin{pmatrix}
 x_1\\ \vdots\\x_r
 \end{pmatrix}
 $$
 be an eigenvector of $B$ with eigenvalue $\lambda$.
 We have
 $$
\lambda x_i=\sum_{j=1}^r|\langle v_j,w\rangle|^{2\alpha-2}\langle v_j,v_i\rangle x_j.
 $$
 for $i=1,2,\dots,r$.
Choose $i$ such that $|x_i|$ is maximal.
Then we have
$$
\lambda|x_i|\leq \sum_{j=1}^r|\langle v_j,w\rangle|^{2\alpha-2}|\langle v_j,v_i\rangle| |x_j|\leq  \sum_{j=1}^r|\langle v_j,w\rangle|^{2\alpha-2}|\langle v_j,v_i\rangle| |x_i|.
$$
If we set $\beta=\alpha/(\alpha-1)$ and $\gamma=\alpha$, then $1/\beta+1/\alpha=1$ and by
H\"older's inequality we get 
\begin{multline*}
[{\bf v}]_{2\alpha}^{2\alpha}\leq \lambda\leq \sum_{j=1}^r|\langle v_j,w\rangle|^{2\alpha-2}|\langle v_j,v_i\rangle|\leq
\Big(\sum_{j=1}^r |\langle v_j,w\rangle|^{(2\alpha-2)\beta}\Big)^{1/\beta}\Big(\sum_{j=1}^r |\langle v_j,v_i\rangle|^\alpha\Big)^{1/\alpha}\leq\\
\Big(\sum_{j=1}^r|\langle v_j,w\rangle|^{2\alpha}\Big)^{(\alpha-1)/\alpha}(\mu_\alpha({\bf v})^\alpha+1)^{1/\alpha}=
 [{\bf v}]_{2\alpha}^{2\alpha-2}(\mu_{\alpha}({\bf v})^\alpha+1)^{1/\alpha}.
\end{multline*}
So we conclude that $[{\bf v}]^{2\alpha}_{2\alpha}\leq \mu_{\alpha}({\bf v})^\alpha+1$.

\end{proof}
\begin{example}
The inequality in Proposition~\ref{prop:comparemeasures} does not hold when $\alpha<1$.
For example,  if ${\bf e}=(e_1,e_2)$ in the (trivial) tensor product space $(\C^2,(\C^2))$, then
we have $\mu_{\alpha}({\bf e})=0$. If $0<\alpha< 1$, then the maximum of
$$
\Big(|\langle e_1,u\rangle|^{2\alpha}+|\langle e_2,u\rangle|^{2\alpha}\Big)^{1/(2\alpha)}
$$
is attained for $u=(e_1+e_2)/\sqrt{2}$. So we have
$$
[{\bf v}]_{2\alpha}^{2\alpha}=2\cdot 2^{-\alpha}=2^{(1-\alpha)/(2\alpha)}>1=1+\mu_{\alpha}({\bf v})^\alpha.
$$
\end{example}


\section{$t$-orthogonality}\label{sec4}
In this section we will discuss a notion of orthogonality for pure tensors that is stronger than the usual notion of orthogonality. 
Recall that an $r$-tuple ${\bf S}=(S_1,\dots,S_r)$ of unit tensors is $t$-orthogonal if $[{\bf S}]_{2/t}=1$.
\begin{lemma}\label{lem:tplusu}
Suppose that ${\bf S}=(S_1,\dots,S_r)$ is $t$-orthogonal $r$-tuple of unit tensors, and ${\bf T}=(T_1,\dots,T_r)$ is an $u$-orthogonal $r$-tuple of unit tensors.
Then ${\bf S}\otimes {\bf T}$ is $(t+u)$-orthogonal. 
\end{lemma}
\begin{proof}
From $[{\bf S}]_{2/t}=1$ and $[{\bf T}]_{2/e}=1$ follows that
$$
1\leq [{\bf S}\otimes {\bf T}]_{2/(t+u)}\leq [{\bf S}]_{2/t}[{\bf T}]_{2/u}\leq 1
$$
by  H\"older's inequality. So $[{\bf S}\otimes {\bf T}]_{2/(t+u)}=1$ and ${\bf S}\otimes {\bf T}$ is $(t+u)$-orthogonal.
\end{proof}
\begin{example}
If $x_1,\dots,x_r$ are orthogonal in $\C^p$, and $y_1,\dots,y_r\in \C^q$ are orthogonal, then  by Lemma~\ref{lem:tplusu}
$$
(x_1\otimes y_1,\dots,x_r\otimes y_r)\in (\C^p\otimes \C^q)^r
$$
is $2$-orthogonal. 
\end{example}

Orthogonality is also stable under taking vertical tensor products.
\begin{lemma}
If ${\bf S}=(S_1,\dots,S_r)$ is an $r$-tuple of unit tensors, ${\bf T}=(T_1,\dots,T_s)$ is an $S$-tuple of unit tensors, 
and ${\bf S}$ and ${\bf T}$ are both $t$-orthogonal, then ${\bf S}\boxtimes {\bf T}$ is also $t$-orthogonal.
\end{lemma}
\begin{proof}
This follows from Proposition~\ref{prop:verticaltensor}.
\end{proof}
Using horizontal and vertical tensor product, we can easily see that  the tensor $M_{p,q,r}$ has a Diagonal Singular Value Decomposition.
\begin{proposition}\label{prop:eijk}
We define
$$
{\bf e}=(e_{j,k}\otimes e_{k,i}\otimes e_{i,j}\mid 1\leq i\leq p, 1\leq j\leq q,1\leq k\leq r)\in (\C^{p\times q}\otimes \C^{q\times r}\otimes \C^{r\times p})^{pqr}.
$$
Then ${\bf e}$ is $2$-orthogonal.
\end{proposition}
\begin{proof}
Let $e_1,\dots,e_p$ denote the orthonormal basis of $\C^p$. Then $(e_1,\dots,e_p)$ is $1$-orthogonal.
So
$$
(e_i\otimes e_i\mid i=1,\dots,p)\in (\C^p\otimes \C^p)^p
$$
is  $2$-orthogonal. The tuples 
\begin{eqnarray*}
P&=& (e_i\otimes 1\otimes e_i\mid 1\leq i\leq p)\in( \C^p\otimes \C\otimes \C^p)^p,\\
Q&=& (e_j\otimes e_j\otimes 1\mid 1\leq j\leq q)\in( \C^q\otimes \C^q\otimes \C)^q,\\
R&=& (1\otimes e_k\otimes e_k\mid 1\leq k\leq r)\in( \C\otimes \C^r\otimes \C^r)^r.
\end{eqnarray*}
are $2$-orthogonal.
We will write $e_{i,j}$ instead of $e_i\boxtimes e_j$.
Now  the tuple
$$
P\boxtimes Q\boxtimes R=(e_{i,j}\otimes e_{j,k}\otimes e_{k,i}\mid1\leq i\leq p,1\leq j\leq q,1\leq k\leq r)\in (\C^{q\times r}\otimes \C^{r\times p}\otimes \C^{p\times q})^{pqr}
$$
is $2$-orthogonal as well. 
\end{proof}
\begin{proof}[Proof of Theorem~\ref{theo:Mpqr}]
This follows immediately from the definition of the diagonal singular value decomposition and Proposition~\ref{prop:eijk}.
\end{proof}
\begin{proof}[Proof of Theorem~\ref{theo:groupalgebra}]
Let $Z_1,\dots,Z_s$ be the irreducible representations of $G$. We have an isomorphism
\begin{equation}\label{eq:groupalgebraiso}
\C G\cong \bigoplus_{i=1}^s \Hom(Z_i,Z_i).
\end{equation}
For $A=\sum_{g\in G}\lambda_g g\in \C G$
we define $A^\star=\sum_{g\in G}\overline{\lambda_g}g^{-1}$.
We may view $A$ as an endomorphism of $\C G$ by left multiplication.
The Hermitian form on $\C G$ is given by
$$
\langle A,B\rangle={\textstyle \frac{1}{n}}\trace(AB^\star).
$$
We can write
$$
e=\sum_{i=1}^s \pi_i
$$
where $\pi_i$ is the projection onto $Z_i$.
The decomposition (\ref{eq:groupalgebraiso}) is orthogonal. 
The multiplication tensor
$$
T_G\in \C G\otimes \C G\otimes \C G
$$
decomposes
$$
T_G=\sum_{i=1}^s T_i
$$
where 
$$T_i\in \Hom(Z_i,Z_i)\otimes \Hom(Z_i,Z_i)\otimes \Hom(Z_i,Z_i)\subseteq \C G\otimes \C G \otimes \C G.
$$
is the tensor for multiplication in $\Hom(Z_i,Z_i)$.
We have
$$
T_i=\sum_{g,h} \pi_i g\otimes h \otimes h^{-1}g^{-1}.
$$
So it follows that
$$
\|T_i\|^2=\frac{1}{n}\sum_{g,h}\trace((\pi_i g)^\star (\pi_i g))=
\frac{1}{n}\sum_{g,h} \trace( \pi_i^\star g^{-1} g \pi_i)=\frac{1}{n}\sum_{g,h}\trace(\pi_i^2)=n\trace(\pi_i)=nd_i^2.
$$
Note that ${\bf T}=(T_1,T_2,\dots,T_s)$ is $3$-orthogonal. The tensor $T_i$ corresponds to matrix multiplication in $\Hom(Z_i,Z_i)$. We can write
$$
T_i=\sum_{j=1}^{d_i^3}\lambda_j a_j\otimes b_j\otimes c_j
$$
where $(a_j\otimes b_j\otimes c_j,1\leq j\leq d_i^3)$ is $2$-orthogonal list of unit vectors. The norm on $\Hom(Z_i,Z_i)$ that
is induced from the norm on $\C G$ may not be the same as the Euclidean norm given by
$A\mapsto\trace(A^\star A)$, but they are the same up to a scalar. This implies that all $\lambda_j$'s are the same. Since $nd_i^2=\|T_i\|^2=d_i^3\lambda_j^2$ we have that  $\lambda_j=\sqrt{\frac{n}{d_i}}$ for all $j$.
So the irreducible representation $Z_i$ contributes the singular value $\sqrt{\frac{n}{d_i}}$ with multiplicity $d_i^3$.

\end{proof}

\begin{proposition}\label{prop:orthogonalbound}
Suppose that  ${\bf V}=(V,(V^{(1)},\dots,V^{(d)}))$ is a tensor product space, $t\geq 1$ and  ${\bf v}=(v_1,\dots,v_r)\in V^r$ is a $t$-orthogonal $r$-tuple of pure tensors of unit length. 
Then we have  $n\leq \dim(V)^{1/t}$.
\end{proposition}
\begin{proof}
Consider the $(2m-1)$-dimensional unit sphere in $\C^{2m}$ given by
$$
|z_1|^2+\cdots+|z_m|^2=1.
$$
Suppose $z$ is a random point on the sphere (with uniform distribution). We will give an estimate for the expectation $\EE(|z|^\alpha)$.
It is clear that $\EE(|z|^\alpha)\geq \EE(|w|^\alpha)$ where $w$ is a random point in the $(2m)$-dimensional ball $B_{2m}$ defined by
$$
|w_1|^2+\cdots+|z_m|^2\leq 1
$$
Let $x=1/\sqrt{2m}$, and let $D$ be the body defined by 
$$|w_1|\leq x\mbox{ and } |w_1|^2+\cdots+|w_m|^2\leq 1
$$
and $E$ be the body defined by
$$
|w_1|\leq x\mbox{ and } |w_2|^2+\cdots+|w_m|^2\leq 1.
$$
Then $D\subseteq E$ and  $E$ is a product of an $(2m-2)$-dimensional ball with radius $1$ and a disk of radius $x$.
We have
$$
\PP(|w|\leq x)=\frac{\vol(D)}{\vol(B_{2m})}\leq \frac{\vol(E)}{\vol(B_{2m})}=\frac{\vol(B_{2m-2})\pi x^2}{\vol(B_{2m})}=\frac{x^2}{m}=\frac{1}{2}
$$
where we use the formula $\vol(B_{2m})=\pi^m/m!$.
It follows that 
$$
\EE(|z|^\alpha)\geq \EE(|w|^{\alpha})\geq x^\alpha\PP(|w|\geq x)\geq\textstyle\frac{1}{2}x^{\alpha}=2^{-1-\alpha/2}m^{-\alpha/2}.
$$
Suppose that $u=u^{(1)}\otimes \cdots\otimes u^{(d)}$ is a fixed unit pure tensor in $V$ and $z=z^{(1)}\otimes \cdots \otimes z^{(d)}$ is a random unit pure tensor.
Let $n=\dim(V)$ and $n_i=\dim(V_i)$ for all $i$.
Then we have
$$
\EE(|\langle z, u\rangle|^{2/t})=\prod_{s=1}^d \EE(|\langle z^{(i)},u^{(i)}\rangle|^{2/t})\geq 2^{-(1+1/t)d}\prod_{i=1}^d n_i^{-1/t}=2^{-(1+1/t)d}n^{-1/t}.
$$
It follows that
$$
1\geq \sum_{i=1}^r \EE(|\langle v_i,z\rangle|^{2/t})=2^{-(1+1/t)d}n^{-1/t}r.
$$
So we get
$$
r\leq 2^{(1+1/t)d}n^{1/t}.
$$
For a positive integer $q$, 
$$v^{\boxtimes q}=\underbrace{v\boxtimes \cdots\boxtimes v}_q$$ is also $t$-orthogonal, and it has $r^q$ vectors in an $n^q$-dimensional vector space.
So we have
$$
r^q\leq 2^{(1+1/t)d} n^{q/t}.
$$
Taking the $q$-th root gives
$$
r\leq 2^{(1+1/t)d/q} n^{1/t}.
$$
Taking the limit $q\to\infty$ yields
$$
r\leq n^{1/t}.
$$
\end{proof}
The following lemma justifies the term $t$-orthogonality.
\begin{lemma}
If  $(v,w)$ is $t$-orthogonal, where $v=v^{(1)}\otimes \cdots \otimes v^{(d)}$ and $w=w^{(1)}\otimes \cdots\otimes w^{(d)}$, then we have $\langle v^{(i)},w^{(i)}\rangle=0$ for at least $t$ values of $i$.
\end{lemma}
\begin{proof}
Suppose that $u=u^{(1)}\otimes \cdots\otimes u^{(d)}$ is a unit pure tensor. Then we have
$$
1\geq \prod_i |\langle v^{(i)}, u^{(i)}\rangle|^{2/t}+\prod_i|\langle w^{(i)},u^{(i)}\rangle|^{2/t}.
$$
Choose $\varepsilon$ with $0<\varepsilon<\sqrt{2}$ and $u^{(i)}$  such that $|\langle v^{(i)},u^{(i)}\rangle| =1-\frac{1}{2}\varepsilon^2$ and $u^{(i)},v^{(i)},w^{(i)}$ are dependent.
If $w^{(i)}$ and $v^{(i)}$ are orthogonal, then $|\langle w^{(i)},u^{(i)}\rangle|=\varepsilon+o(\varepsilon)$,
because $|\langle v^{(i)},u^{(i)}\rangle|^2+|\langle w^{(i)},u^{(i)}\rangle|^2=1$. If $w^{(i)}$ and $v^{(i)}$ are not orthogonal,
then $|\langle w^{(i)},u^{(i)}\rangle|=|\langle w^{(i)},v^{(i)}\rangle|+o(\varepsilon)$.
If $s$ is the number of $i$ for which $v^{(i)}$ and $w^{(i)}$ are orthogonal, then we have
\begin{multline*}
1\geq  \prod_i |\langle v^{(i)}, u^{(i)}\rangle|^{2/t}+\prod_i|\langle w^{(i)},u^{(i)}\rangle|^{2/t}=(1-{\textstyle \frac{1}{2}}\varepsilon^2)^{2/t}+|C\varepsilon^{s}+o(\varepsilon^s)|^{2/t}=\\=
 1-{\textstyle \frac{d}{t}}\varepsilon^2+o(\varepsilon^2)+C^{2/t}\varepsilon^{2s/t}+o(\varepsilon^{2s/t}).
 \end{multline*}
for some constant $C$. We must have $s/t\geq 1$, otherwise the inequality is not satisfied for small $\varepsilon$.
\end{proof}

\begin{example}
Consider the following triple of pure tensors
$$
{\bf e}=(e_1\otimes e_1\otimes e_1,e_1\otimes e_2\otimes\otimes e_2,e_2\otimes e_1\otimes e_2)\in (\C^2\otimes \C^2\otimes \C^2)^3
$$
Then every pair of vectors of ${\bf e}$ is $2$-orthogonal. However, ${\bf e}$ itself is not $2$-orthogonal, because
it violates Proposition~\ref{prop:orthogonalbound}:
$$
n=3>8^{1/2}=\dim(V)^{1/2}.
$$
\end{example}
\begin{example}
Let $G$ be a group of order $n$, $\C G$ be its group algebra and consider the tensor product space $\C G\otimes \C G\otimes \C G$.
Let
$$
{\bf v}=(g\otimes h\otimes k\mid g,h,k\in G;\ ghk=1).
$$
be a list of $n^2$ vectors. We claim that  ${\bf v}$ is $\frac{3}{2}$-orthogonal. Suppose that
$$
w=\Big(\sum_{g\in G}a_g g\Big)\otimes \Big(\sum_{g\in G} b_g g\Big)\otimes \Big(\sum_{g\in G} c_g g\Big) 
$$
is a pure tensor with $\sum_{g\in G}|a_g|^2=\sum_{g\in G}|b_g|^2=\sum_{g\in G}|c_g|^2=1$.
Using the inequality $pqr\leq \frac{1}{3}(p^3+q^3+r^3)$ we get
\begin{multline*}
\sum_{g,h,k\in G\atop gkh=1} |\langle g\otimes h\otimes k,w\rangle|^{4/3}=\sum_{g,h,k\in G\atop ghk=1} |a_gb_hc_k|^{4/3}=
\sum_{g,h,k\in G\atop ghk=1} |a_gb_h|^{2/3} |b_hc_k|^{2/3}|c_ka_g|^{2/3}\leq\\\leq 
\frac{1}{3}\sum_{g,h,k\in G\atop ghk=1}  (|a_gb_h|^2+|b_hc_k|^2+|c_ka_g|^2)=1.
\end{multline*}
To see the last equality, note that
$$
\sum_{g,h,k\in G\atop ghk=1}|a_gb_h|^2=\sum_{g\in G}|a_g|^2\sum_{h\in G}|b_h|^2=1\cdot 1=1.
$$
This proves that ${\bf v}$ is $\frac{3}{2}$-orthogonal. For $t>\frac{3}{2}$, ${\bf v}$ cannot be $t$-orthogonal
because otherwise this would violate Proposition~\ref{prop:orthogonalbound}.
\end{example}

\section{Lower bounds for the nuclear norm}\label{sec5}
\begin{proof}[Proof of Theorem~\ref{theo:lowbound}]
Suppose that $\alpha\geq 1$, $T$ is a tensor, and ${\bf S}=(S_1,\dots,S_r)$ is an $r$-tuple of tensors.
We can write
$$
T=\sum_{i=1}^s \mu_j w_j
$$
where $\mu_1,\dots,\mu_s$ are positive real numbers such that  $\sum_{j=1}^s \mu_j=\|T\|_\star$ and $w_1,\dots,w_s$ are pure unit tensors.
Define
$$
z_j:=\begin{pmatrix}
|\langle w_j,S_1\rangle|\\
\vdots\\
|\langle w_j,S_r\rangle|
\end{pmatrix}\in \C^r
$$
for $j=1,2,\dots,s$. We have 
$$\|z_j\|_\alpha=\Big(\sum_{i=1}^r | \langle w_j,S_i\rangle|^\alpha\Big)^{1/\alpha}\leq [{\bf S}]_\alpha.
$$
for all $j$.
It follows that
\begin{multline*}
\Big(\sum_{i=1}^r |\langle T,S_i\rangle|^\alpha\Big)^{1/\alpha}\leq 
\Big(\sum_{i=1}^r \Big(\sum_{j=1}^s \mu_j |\langle w_j,S_i\rangle|\Big)^\alpha\Big)^{1/\alpha}=\Big\|\sum_{j=1}^s \mu_j z_j\Big\|_{\alpha}\leq\\ \leq \sum_{j=1}^s |\mu_j|\|z_j\|_\alpha\leq\sum_{j=1}^s |\mu_j|[{\bf S}]_\alpha=\|T\|_\star [{\bf S}]_\alpha.
\end{multline*}
\end{proof}

For a permutation $\sigma\in \SSS_n$, define 
$$e_\sigma=e_{\sigma(1)}\otimes e_{\sigma(2)}\otimes\cdots\otimes e_{\sigma(n)}\in (\C^n)^{\otimes n}
$$
and 
$$
{\bf e}=(e_\sigma,\sigma\in \SSS_n).
$$
We now study the determinant tensor $\sum_\sigma \sgn(\sigma) e_\sigma$ and the permanent tensor $\sum_\sigma e_\sigma$.
\begin{proof}[Proof of Theorem~\ref{theo:det}]
If $a^{(1)},\cdots,a^{(n)}$ are vectors of unit length, then Hadamard's inequality yields
$$
|\langle {\textstyle \det_n},a^{(1)}\otimes a^{(2)}\otimes \cdots\otimes a^{(n)}\rangle|=|\det(a^{(1)},\dots,a^{(n)})|\leq \|a^{(1)}\|\cdots \|a^{(n)}\|=1.
$$
Therefore, we have $[{\textstyle \det_n}]\leq 1$.
It follows from Corollary~\ref{cor:spectralbound} that
$$
\|{\textstyle \det_n}\|_\star\geq \|{\textstyle \det_n}\|_\star [{\textstyle \det_n}]\geq \|{\textstyle \det_n}\|^2=n!\ .
$$
\end{proof}
The following theorem proven in~\cite{CEL} is the permanent analog of Hadamard's inequality.
\begin{theorem}\label{theo:CEL}
For vectors $a^{(1)},\dots,a^{(n)}\in \C^n$ we have
$$
|\per(A)|\leq \frac{n!}{n^{n/2}}\|a^{(1)}\|\|a^{(2)}\|\cdots\|a^{(n)}\|.
$$
\end{theorem}
\begin{proof}[Proof of Theorem~\ref{theo:perm}]
For vectors $a^{(1)},\cdots,a^{(n)}$ of unit length,we get
$$
|\langle {\textstyle \per_n},a^{(1)}\otimes a^{(2)}\otimes \cdots\otimes a^{(n)}\rangle|=|\per(a^{(1)},\dots,a^{(n)}|\leq \frac{n!}{n^{n/2}}\|a^{(1)}\|\cdots \|a^{(n)}\|=\frac{n!}{n^{n/2}}.
$$
So we have
$[{\textstyle \per_n}]\leq \frac{n!}{n^{n/2}}.$
From Corollary~\ref{cor:spectralbound} follows that
$$
\frac{n!}{n^{n/2}}\|{\textstyle \per_n}\|_\star\geq \|{\textstyle \per_n}\|_\star [{\textstyle \per_n}]\geq \|{\textstyle \per_n}\|^2=n!.
$$
We conclude  that
$\|\per_n\|_\star\geq n^{n/2}$.
\end{proof}

[

\section{The Diagonal Singular Value Decomposition}\label{sec6}
For an $r$-tuple ${\bf v}=(v_1,\dots,v_r)$ and $k<r$ we write ${\bf v}^{[k]}$ for $(v_1,\dots,v_k)$.
We start with the most general, main theorem.
\begin{theorem}\label{theo:main}
Suppose that $V$ is a tensor product space, ${\bf v}=(v_1,\dots,v_r)$ and ${\bf w}=(w_1,\dots,w_s)$
consists of pure tensors in $V$ of unit length, 
 $\lambda_1\geq \lambda_2\geq \cdots\geq\lambda_s>0$, 
 $\sigma_1\geq \sigma_2\geq \cdots \geq \sigma_r>0$
 and
$$
\sum_{i=1}^s\lambda_i w_i=\sum_{j=1}^r\sigma_jv_j.
$$
Also, suppose that $k\leq s$, $l\leq r$ such that $0\leq \delta\leq [{\bf w}^{[k]}]_1$ where $\delta:=k[{\bf v}]_1-l[{\bf w}^{[k]}]_1$. Then we have
$$
[{\bf w}^{[k]}]_1(\sigma_1+\sigma_2+\cdots+\sigma_l)+\delta \sigma_{l+1}\geq (1-\mu_1({\bf w}))(\lambda_1+\lambda_2+\cdots+\lambda_k),
$$
Here we use the conventions that  $0=\lambda_{s+1}=\lambda_{s+2}=\cdots$ and $0=\sigma_{r+1}=\sigma_{r+2}=\cdots$.
\end{theorem}
\begin{proof}
Let 
$$
T=\sum_{i=1}^s\lambda_iw_i=\sum_{j=1}^r \sigma_j v_j.
$$
We have
\begin{multline*}
\sum_{i=1}^k\sum_{j=1}^s \lambda_i|\langle w_i, w_j\rangle|=
\sum_{i=1}^k\sum_{j=1}^k\lambda_i|\langle w_i, w_j\rangle |+\sum_{i=1}^k\sum_{j=k+1}^s\lambda_i|\langle w_i, w_j\rangle|=\\=
\sum_{i=1}^k\sum_{j=1}^k\lambda_j|\langle  w_i, w_j\rangle|+\sum_{i=1}^k\sum_{j=k+1}^s\lambda_i|\langle w_i,w_j\rangle|\geq\\\geq
 \sum_{i=1}^k\sum_{j=1}^k\lambda_j|\langle w_i, w_j\rangle |+\sum_{i=1}^k\sum_{j=k+1}^s\lambda_j|\langle w_i,w_j\rangle|=\sum_{i=1}^k\sum_{j=1}^s \lambda_j|\langle w_i, w_j\rangle|.
\end{multline*}
because $\lambda_i\geq \lambda_j$ whenever $i\geq j$.
Using this, we get
\begin{multline*}
\sum_{i=1}^k |\langle w_i, T\rangle |=\sum_{i=1}^k\Big|\sum_{j=1}^s\lambda_j\langle w_i, w_j\rangle \Big|\geq 
\sum_{i=1}^k \lambda_i-\sum_{i=1}^k\sum_{j\neq i}\lambda_j|\langle w_i, w_j\rangle|=\\=
2\sum_{i=1}^k\lambda_i-\sum_{i=1}^k\sum_{j=1}^s\lambda_j |\langle w_i, w_j\rangle|\geq
2\sum_{i=1}^k\lambda_i-\sum_{i=1}^k\sum_{j=1}^s \lambda_i |\langle w_i, w_j\rangle|=\\=
\sum_{i=1}^k\lambda_i\Big(2-\sum_{j=1}^s |\langle w_i,w_j\rangle |\Big)\geq 
(1-\mu_1({\bf w})) \sum_{i=1}^k \lambda_i.
\end{multline*}

Let $y_{i,j}=|\langle w_i,v_j\rangle|$ if $1\leq i\leq s$ and $1\leq j\leq r$.
We have
$$
(1-\mu_1({\bf w}))\sum_{i=1}^k\lambda_i\leq \sum_{i=1}^k |\langle w_i, T
\rangle|\leq \sum_{i=1}^k\sum_{j=1}^r \sigma_{j}y_{i,j}
=
 \sum_{j=1}^r\sigma_j\sum_{i=1}^k y_{i,j}=\sum_{j=1}^r \sigma_j x_j,$$
where $x_j=\sum_{i=1}^ky_{i,j}\leq [{\bf w}^{[k]}]_1$. We also have 
$$x_1+\cdots+x_r=\sum_{j=1}^r\sum_{i=1}^ky_{i,j}=\sum_{i=1}^k\sum_{j=1}^r y_{i,j}\leq \sum_{i=1}^k [{\bf v}]_1=k[{\bf v}]_1.
$$
If we maximalize the functional $\sum_{j=1}^r \sigma_jx_j$ under the constraints $0\leq x_i\leq [{\bf w}^{[k]}]_1$ for $i=1,2,\dots,l$
 and $x_1+\cdots+x_{r}\leq k[{\bf v}]_1$,
then an optimal solution is $x_1=x_2=\cdots=x_l=[{\bf w}^{[k]}]_1$, $x_{l+1}=k[{\bf v}]_1-l[{\bf w}^{[k]}]_1=\delta$ and $x_{l+2}=\cdots=x_r=0$, and the optimal value is
$$[{\bf w}^{[k]}]_1(\sigma_1+\cdots+\sigma_l)+\delta\sigma_{l+1}.
$$
\end{proof}
The following result gives a lower bound for the nuclear norm: 
\begin{theorem}\label{theo:nuclearlowerbound}
If  ${\bf w}=(w_1,\dots,w_s)$ is an orthogonal $r$-tuple of pure tensors of unit length, $\lambda_1\geq \cdots\geq \lambda_s>0$ and
$T=\sum_{i=1}^s \lambda_i w_i$, then we have
$$
\|T\|_\star\geq\frac{\sum_{i=1}^k \lambda_i}{[{\bf w}^{[k]}]_1}
$$
\end{theorem}
\begin{proof}
We can write $T=\sum_{i=1}^r \sigma_i v_i$ where $v_i$ is a pure tensor of unit length for all $i$,
$\sigma_1\geq \sigma_2\geq \cdots \geq \sigma_r>0$ and $\|T\|_\star=\sum_{i=1}^r\sigma_i$.
We have
$$
\sum_{i=1}^s \lambda_i w_i=\sum_{j=1}^r \sigma_jv_j
$$
and $\mu_1({\bf w})=0$ because ${\bf w}$ is orthogonal. From Theorem~\ref{theo:main} follows that
$$
[{\bf w}^{[k]}]_1(\sigma_1+\cdots+\sigma_r)\geq \lambda_1+\cdots+\lambda_s.
$$
\end{proof}


\begin{proof}[Proof of Theorem~\ref{theo:uniquesingval}]
Suppose that ${\bf v}=(v_1,\dots,v_r)$ and ${\bf w}=(w_1,\dots,w_s)$ are $2$-orthogonal tuples
of pure tensors of unit length, and
$$
\sum_{i=1}^s\lambda_iw_i=\sum_{j=1}^r \sigma_j v_j,
$$
such that $\lambda_1\geq \cdots \geq \lambda_s>0$ and $\sigma_1\geq \cdots \geq\sigma_r>0$.
We apply Theorem~\ref{theo:main} with $[{\bf v}]_1=[{\bf w}^{[k]}]_1=1$, $l=k$ and get
$$
\sum_{j=1}^k\sigma_j\geq\sum_{i=1}^k \lambda_i
$$
for all $k$. If we switch the roles of the $v$'s and $w$'s we also get  inequalities in the other directions as well.
We conclude that $r=s$ and $\lambda_i=\mu_i$ for all $i$.
\end{proof}
\begin{proof}[Proof of Theorem~\ref{theo:properties}]
Suppose that the diagonal singular value decomposition of $T$ is
$$
T=\sum_{i=1}^r \sigma_i v_i.
$$
Then we have $\|T\|_\star\leq \sum_{i=1}^r \sigma_i$.
If we take $k=r$, and $\lambda_i=\sigma_i$  in Theorem~\ref{theo:nuclearlowerbound}
then we get
$$
\|T\|_\star\geq\frac{ \sum_{i=1}^r \sigma_i}{[{\bf v}]_1}=\sum_{i=1}^r\sigma_i,
$$
so we conclude that $\|T\|_\star=\sum_{i=1}^r\sigma_i$.

Since ${\bf v}=(v_1,\dots,v_r)$ is 2-orthogonal, we have
$$
\|T\|^2=\sum_{i=1}^r \sigma_i^2
$$
and
$[{\bf v}]_1=1$.

If $u$ is a pure tensor of unit length, then
$$
|\langle T,u\rangle|=\Big|{\textstyle \sum_{i=1}^r \sigma_i\langle v_i,u\rangle}\Big|\leq \sigma_1\sum_{i=1}^r |\langle v_i,u\rangle|\leq \sigma_1[{\bf v}]_1=\sigma_1.
$$
Clearly $\langle T,v_1\rangle=\sigma_1$. So we conclude that $[T]=\sigma_1$.
\end{proof}
\begin{proof}[Proof of Theorem~\ref{theo:distinctsingvalues}]
Suppose that $T$ has a diagonal singular value decomposition with singular values $\sigma_1>\dots>\sigma_r>0$. 
We can write $T=\sum_{j=1}^r \sigma_jv_j$. Suppose that we have another singular value decomposition
$T=\sum_{i=1}^r\sigma_i w_i$. (Note that the singular values are determined by $T$ because of Theorem~\ref{theo:uniquesingval}).

Let $y_{i,j}=|\langle w_i,v_i\rangle|$.
Then $\sum_{j=1}^r y_{i,j}\leq [{\bf v}]_1=1$
and $\sum_{i=1}^r y_{i,j}leq [{\bf w}]_1=1$.
Fix $k\leq r$ and let $x_j=\sum_{i=1}^k y_{i,j}\leq 1$. From the proof of Theorem~\ref{theo:main} follows that
$$
x_1+\cdots+x_r\leq k.
$$
and
$$
\sum_{i=1}^k \sigma_i\leq \sum_{j=1}^r \sigma_jx_j.
$$
Since $\sigma_1,\dots,\sigma_r$ are distinct, we must have $x_1=x_2=\cdots =x_k=1$ and $x_{k+1}=\cdots=x_r=0$.
This implies that $y_{i,j}=0$ if $i\leq k$ and $j\geq k+1$.
So $y_{i,j}=0$ for $i<j$ and by symmetry, $y_{i,j}=0$ for $i>j$.
This proves that $|\langle v_i,w_i\rangle|=y_{i,i}=1$ for all $i$.
So $w_i$ is equal to $v_i$ up to a unit scalar, say $w_i=\gamma_iv_i$. It follows that
$$
\sum_{i=1}^r \sigma_i v_i=\sum_{i=1}^r \sigma_iw_i=\sum_{i=1}^r\sigma_i\gamma_iv_i
$$
and because $v_1,\dots,v_r$ are linearly independent, it follows that $\gamma_i=1$ and $w_i=v_i$ for all $i$.

\end{proof}

\begin{proof}[Proof of Theorem~\ref{theo:torthoSVD}]
Suppose that $T$ is a tensor with 2 diagonal singular value decompositions
$$
T=\sum_{i=1}^r \sigma_iw_i=\sum_{j=1}^r \sigma_jv_j
$$
with $\sigma_1\geq \cdots\geq \sigma_r>0$, and that ${\bf w}=(w_1,\dots,w_r)$ is $t$-orthogonal with $t>2$.
Let $y_{i,j}=|\langle w_i,v_j\rangle|$.

From the proof of Theorem~\ref{theo:main} follows that
$$
\sum_{i=1}^r y_{i,j}=1.
$$
for all $j$.
We also have 
$$
\sum_{i=1}^r y_{i,j}^\alpha \leq 1.
$$
where $\alpha=2/t<1$, because ${\bf w}$ is $t$-orthogonal.
Subtracting gives
$$
\sum_{i=1}^r y_{i,j}^{\alpha}(1-y_{i,j}^{1-\alpha})\leq 0.
$$
It follows that $y_{i,j}\in \{0,1\}$ for all $i,j$.
 The column sums of $Y=(y_{i,j})$ are $1$. So every
column has exactly one $1$. So the matrix has exactly $r$ $1$'s. Since the row sums are also $1$, it follows that
every row has exactly one 1 as well. So $Y$ is a permutation matrix.
There exists a permutation $\phi$ of $\{1,2,\dots,r\}$ such that
$$
v_i=\gamma_i w_{\phi(i)}.
$$
where $\gamma_i$ is a unit for all $i$.
We have
$$
\sum_{i=1}^r \sigma_i  v_i=\sum_{i=1}^r \gamma_i\sigma_i w_{\phi(i)}=\sum_{i=1}^r\sigma_{\phi(i)}w_{\phi(i)}.
$$
Since ${\bf w}$ is linearly independent, it follows that $\gamma_i\sigma_i=\sigma_{\phi(i)}$ for all $i$.
So $\gamma_i=1$ and $\sigma_i=\sigma_{\phi(i)}$ for all $i$.
This shows that
$$
\sigma_1v_1,\dots,\sigma_rv_r
$$
is a permutation of 
$$
\sigma_1w_1,\dots,\sigma_rw_r.
$$
So the diagonal singular value decomposition is unique.
\end{proof}
\section{Tensors without a diagonal singular value decomposition}\label{sec7}

\begin{example}
Consider the permanent $\per_n$. Suppose that it has a DSVD and that its singular values are $\sigma_1,\dots,\sigma_r$.
Then we have 
$$[\per_n]=\sigma_1=\frac{n!}{n^{n/2}},\quad
\|\per_n\|^2=n!, \quad \|\per_n\|_\star=n^{n/2}.
$$
We have
$$
\sigma_1\sum_{i=1}^r \sigma_i=[\per_n]\|\per_n\|_\star=n!=\|\per_n\|^2=\sum_{i=1}^r \sigma_i^2
$$
so it follows that $\sigma_1=\sigma_2=\cdots=\sigma_r$.
So 
$$\frac{n^n}{n!}=\frac{\|\per_n\|_\star}{[\per_n]}=\frac{r\sigma_1}{\sigma_1}=r.
$$
For $n\geq 3$, $n^n/n!$ is not an integer (the denominator is divisible by $n-1$), so $\per_n$ cannot have a diagonal singular value decomposition.
\end{example}
\begin{example}
Consider the determinant $\textstyle\det_n$. Suppose that $\det_n$ has a DSVD. A similar argument as in the previous example shows that $\det_n$ has a singular value $\sigma$ with multiplicity $r$, where
$$
r=\frac{\|\det_n\|_\star}{[\det_n]}=n!.
$$
So there exists a $2$-orthogonal $r$-tuple of pure tensors of unit length. This implies that $r\leq n^{n/2}$ by Proposition~\ref{prop:orthogonalbound}. For $n\geq 3$ we have $n!>n^{n/2}$, so
$\det_n$ cannot have a diagonal singular value decomposition.
\end{example}
\section{Appendix: The tensor rank of the determinant and the permanent}\label{sec8}
For a subset $I=\{i_1,i_2,\dots,i_r\}\subseteq \{1,2,\dots,n\}$ with $i_1<\cdots<i_r$ define
$$
{\textstyle \det_r}(I)=\sum \sgn(\sigma) e_{i_{\sigma(1)}}\otimes \cdots \otimes e_{i_{\sigma(r)}},
$$
$$
\sgn(I)=(-1)^{\textstyle i_1+\cdots+i_r-{r+1\choose 2}}.
$$
and
$$
I^c=\{1,2,\dots,n\}\setminus I.
$$

We have the following generalized Laplace expansion
$$
{\textstyle \det_n}=\sum_{I}\sgn(I)\textstyle \det_r(I)\otimes \det_{n-r}(I^c)
$$
where $I$ runs over all ${n\choose r}$ subsets of $\{1,2,\dots,n\}$ with cardinality $r$.

By flattening, we may view the tensor $\det_n$ as a tensor in $\C^{n^r}\otimes \C^{n^{n-r}}$.
The tensors $\det_r(I)$ where $I$ is a subset of $\{1,2,\dots,n\}$ with $r$ elements are linearly independent.
The tensors $\det_r(I^c)$ are linearly independent as well. This shows that
the flattened tensor has rank at least ${n\choose r}$. So we have
$$
\rank({\textstyle  \det_n})\geq {n\choose r}.
$$
We get the best lower bound if $r=\lfloor n/2\rfloor$:
$$
\rank({\textstyle \det_n})\geq {n\choose \lfloor n/2\rfloor}.
$$
We have a similar Laplace expansion for the permanent, so we also get
$$
\rank({\textstyle \per_n})\geq {n\choose \lfloor n/2\rfloor}.
$$
So the ranks of the determinant and permanent grow at least exponentially. 
We also have an exponential lower bound for the permanent. An exponential upper bound
for the rank of the determinant seems not to be known. However, the
obvious bound $\rank(\det_n)\leq n!$ is not sharp for $n\geq 3$.

For $n=3$, we have
\begin{multline*}
{\textstyle \det_3}={\textstyle \frac{1}{2}}\Big((e_3+e_2)\otimes (e_1-e_2)\otimes (e_1+e_2)+(e_1+e_2)\otimes (e_2-e_3)\otimes (e_2+e_3)+
2e_2\otimes (e_3-e_1)\otimes (e_3+e_1)+\\+
(e_3-e_2)\otimes (e_2+e_1)\otimes (e_2-e_1)+(e_1-e_2)\otimes (e_3+e_2)\otimes (e_3-e_2)\Big).
\end{multline*}
So $\rank(\det_3)\leq 5$. Zach Teitler pointed out that this implies that the Waring rank of a $3\times 3$ matrix is at most 20. He also pointed out that one can show that 
$\rank(\det_3)\geq 4$.
If $n>3$,  then we can again use the generalized Laplace expansion
$$
{\textstyle \det_n}=\sum_{I}\sgn(I)\textstyle \det_3(I)\otimes \det_{n-3}(I^c).
$$
where $I$ runs over all subsets of $\{1,2,\dots,n\}$ with 3 elements.
This proves that
$$
\rank({\textstyle \det_n})\leq {n\choose 3}\rank({\textstyle \det_{n-3}})\rank({\textstyle \det_3})\leq \frac{5\cdot n!}{6\cdot (n-3)!}
$$
We can rewrite this as
$$
\frac{\rank(\det_n)}{n!}\leq \frac{5}{6}\cdot \frac{\rank(\det_{n-3})}{(n-3)!}.
$$
By induction, we get
$$
\rank({\textstyle \det_n})\leq\Big(\frac{5}{6}\Big)^{\textstyle \lfloor \frac{n}{3}\rfloor} \cdot n!.
$$

\begin{remark}
Homogeneous polynomials can be thought of as symmetric tensors. For symmetric
tensors there is also a notion of rank, the so-called {\em symmetric rank}. The symmetric rank is
different from, but closely related to the tensor rank. The determinant and permanent can be thought
of as homogeneous polynomials. Lower bounds for the symmetric tensor rank
of the determinant and permanent can be found in \cite{LT} and \cite{Shafiei}.
\end{remark}
\subsection*{Acknowledgment} The author thanks Zach Teitler for useful comments and a correction.

\ \\[20pt]
\noindent{\sl Harm Derksen\\
Department of Mathematics\\
University of Michigan\\
530 Church Street\\
Ann Arbor, MI 48109-1043, USA\\
{\tt hderksen@umich.edu}}

 \end{document}